\documentclass[11pt]{amsart}
\usepackage{}
\usepackage{palatino}
\usepackage{color}

\newcommand{\p}{\partial}

\newcommand{\FF}{\mathbb{F}}

\newcommand{\NN}{\mathbb{N}}

\newcommand{\cO}{\mathcal{O}}

\newcommand{\fg}{\mathfrak{g}}

\newcommand{\fl}{\mathfrak{l}}
\newcommand{\fm}{\mathfrak{m}}

\newcommand{\ft}{\mathfrak{t}}

\DeclareMathOperator{\Aut}{Aut}

\DeclareMathOperator{\Char}{char}

\DeclareMathOperator{\Der}{Der}

\DeclareMathOperator{\GL}{GL}

\DeclareMathOperator{\id}{id}

\DeclareMathOperator{\Lie}{Lie}
\DeclareMathOperator{\Mat}{Mat}

\DeclareMathOperator{\C}{C}

\DeclareMathOperator{\N}{N}
\DeclareMathOperator{\Ind}{Ind}
\DeclareMathOperator{\dv}{div}

\DeclareMathOperator{\Tor}{Tor}

\usepackage{pictex}
\usepackage{amscd}
\usepackage{latexsym}
\usepackage{amssymb}
\usepackage{eucal}
\usepackage{eufrak}
\usepackage{verbatim}
\usepackage{hyperref}

\numberwithin{equation}{section}
\newtheorem{Theorem}{Theorem}[section]
\newtheorem{Lemma}[Theorem]{Lemma}
\newtheorem{Corollary}[Theorem]{Corollary}
\newtheorem{Proposition}[Theorem]{Proposition}

\theoremstyle{Theorem}

\newtheorem*{thm*}{Theorem}
\newtheorem*{thm**}{Corollary}
\newtheorem*{thm***}{Theorem B}
\theoremstyle{remark}

\numberwithin{equation}{section}

\begin{document}
\title[Semisimple Orbits]{On the semisimple orbits of restricted Cartan type Lie algebras $W,~S$ and $H$}
\author[Hao Chang \lowercase{and} Ke Ou]{Hao Chang \lowercase{and} Ke Ou*}
\address[Hao Chang]{School of Mathematics and Statistics, Central China Normal University, 430079 Wuhan, People's Republic of China}
\email{chang@ccnu.edu.cn}
\address[Ke Ou]{School of Statistics and Mathematics, Yunnan University of Finance and Economics, 650221 Kunming, People's Republic of China}
\email{keou@ynufe.edu.cn}
\date{\today}
\thanks{* Corresponding author.}

\makeatletter
\makeatother
\begin{abstract}
In this short note,
we give a description of semisimple orbits in
the restricted Cartan type Lie algebras $W, S, H$.
\end{abstract}
\maketitle
\section{Introduction}
Let $(\fg,[p])$ be a restricted Lie algebra with connected automorphism group $G_\fg:=\Aut_p(\fg)^{\circ}$.
The algebraic group $G_\fg$ acts naturally on the constructible set $S_\fg$ of semisimple elements of $\fg$.
A basic problem is to understand the set $S_\fg/G_\fg$ of semisimple $G_\fg$-orbits.

In the classical case,
where $\fg:=\Lie(G)$ is the Lie algebra of a connected reductive group $G$,
all maximal tori of $\fg$ are $G$-conjugate (hence $G_\fg$-conjugate) and there is a bijective correspondence $S_\fg/G\rightarrow\ft/W$,
where $\ft$ is a maximal torus and  $W$ its corresponding Weyl group (cf. \cite[(7.12)]{Ja}).
If $\fg$ is not an algebraic Lie algebra,
then maximal tori are not necessarily $G_\fg$-conjugate.
In fact,
for the non-classical simple Lie algebras (which,
by the classification theorem of Block-Wilson-Premet-Strade (cf. \cite{St}),
are of Cartan type provided that the characteristic of $k$ is larger than $5$),
the maximal tori are not all conjugate under the action
of the automorphism group (\cite[Chapter 7]{St}).

In this paper, we study Lie algebras of Cartan type $\fg:=W,S,H$.
In these cases, $\fg$ possesses finitely many $G_\fg$-conjugacy classes of maximal tori.
These algebras have a natural filtration
$$\fg=\fg_{(-1)}\supseteq\fg_{(0)}\supseteq\cdots\supseteq\fg_{(s)}\supseteq(0)$$
by $[p]$-stable subspaces.
Let $\langle x\rangle_p$ denote the torus generated by a semisimple element $x\in S_\fg$.
We define a function
\begin{align}\label{index formula introduction}
\Ind_{\fg}: S_{\fg}\rightarrow\NN_0;~x\mapsto\dim_k\langle x\rangle_p/(\langle x\rangle_p\cap\fg_{(0)}),
\end{align}
whose fibers are $G_\fg$-stable.
Given a maximal torus $\ft_\fg$ of $\fg$,
we consider the Weyl group $W(\fg,\ft_\fg)$ relative to $\ft_\fg$,
defined via $W(\fg,\ft_\fg):=\N_{G_\fg}(\ft_\fg)/\C_{G_\fg}(\ft_\fg)$,
where $\N_{G_\fg}(\ft_\fg)$ and $\C_{G_\fg}(\ft_\fg)$ are the normalizer and the centralizer of $\ft_\fg$ in $G_\fg$,
respectively.
Using basic results on tori,
due to Demushkin \cite{De1, De2},
every maximal torus $\ft_\fg\subseteq\fg$ has the same dimension $\mu(\fg)$.
Moreover,
up to conjugacy,
every integer $0\leq r\leq\mu(\fg)$ gives rise to a unique maximal torus $\ft_{\fg,r}$
such that $\Ind_\fg(x)\leq r$ for all $x\in\ft_{\fg,r}$ and $\ft_{\fg,r}^r:=\Ind_\fg^{-1}(r)\cap\ft_{\fg,r}\neq\emptyset$.
The $W(\fg,\ft_{\fg,r})$-orbits on $\ft_{\fg,r}$ are distinguished by the values of invariant functions,
and the invariants were determined by the second author in \cite[Proposition 3.2]{Ou}. Actually, in the case $r=\mu(\fg)$, the isomorphism $W(\fg,\ft_{\fg,\mu(\fg)})\cong\GL_{\mu(\fg)}(\FF_p)$ was established in \cite{Pre} and \cite{BFS}, and the invariant functions on $ \ft_{\fg,\mu(\fg)} $ under $ \GL_{\mu(\fg)}(\FF_p) $ action were determined in a classical work of L. Dickson \cite{Di}.

The main result reads:
\begin{thm*}
Let $r\in\{0,1,\dots,\mu(\fg)\}$.
Then there is a bijective correspondence $$\Ind_\fg^{-1}(r)/G_\fg\rightarrow\ft_{\fg,r}^r/W(\fg,\ft_{\fg,r}).$$
\end{thm*}
More details refer to Theorem \ref{main theorem}.
We will give description the quotients $\ft_{\fg,r}^r/W(\fg,\ft_{\fg,r})$ by employing $p$-polynomials in Proposition \ref{main prop} respectively.

\emph{Throughout this paper, $k$ denotes an algebraically closed field of characteristic $\Char(k)=:p>3$.}

\bigskip
\noindent
\textbf{Acknowledgments.} This work is supported by NSFC (No. 11801204),
NSF of Yunnan Province (No. 2020J0375),
the Fundamental Research Funds of YNUFE (No. 80059900196).
We are indebted to the referee for
carefully reading the manuscript and providing numerous comments.

\bigskip
\section{Preliminaries}
\subsection{Cartan type Lie algebras type $W,S,H$}\label{section Cartan type lie algebras}
Let $A(n):=k[X_1,\dots,X_n]/(X_1^p,\dots,X_n^p)$ be the truncated polynomial ring in $n$ variables.
We write $x_i$ for the image of $X_i$ in $A(n)$.
Note that $A(n)$ is a finite-dimensional local algebra,
with maximal ideal $\fm:=\langle x_1,\dots,x_n\rangle$.
The Lie algebra $W(n):=\Der(A(n))$ is called the \textit{$n$-th Jacobson-Witt algebra.}
It is an $A(n)$-module in an obvious way,
and has a standard basis $\{x_1^{\alpha_1}\cdots x_n^{\alpha_n}\partial_i;~0\leq\alpha_j<p, 1\leq i\leq n\}$
where $\partial_i$ denotes the partial derivative with respect to the variable $x_i$.

Define the linear map $\dv:W(n)\rightarrow A(n)$ by
\begin{align*}\label{div map}
\dv(\p)=\sum\limits_{i=1}^{n}\p_i(\p(x_i)).
\end{align*}
The Lie algebra $S(n)$ is defined via $S(n):=\{\p\in W(n);~\dv(\p)=0\}$
and the derived algebra $S(n)^{(1)}$ is called \textit{special algebra}.
If $n\geq 3$,
then $S(n)^{(1)}$ is restricted and simple.

Let us move on to the family $H(2m)$.
For $i\in\{1,\ldots,2m\}$, we put
\begin{equation*}\label{definition of sigma i}
\sigma(i):=\left\{
\begin{aligned}
1,~&1\leq i\leq m, \\
-1,~&m+1\leq i\leq 2m.
\end{aligned}
\right.
\end{equation*}
In addition, we define
\begin{equation*}\label{definition of i'}
i':=\left\{
\begin{aligned}
i+m,~& 1\leq i\leq m, \\
i-m,~& m+1\leq i\leq 2m.
\end{aligned}
\right.
\end{equation*}
Let $H(2m):=\{\sum\limits_{i=1}^{2m}f_i\p_i\in W(2m);~\sigma(i)\p_{j'}(f_i)=\sigma(j)\p_{i'}(f_j)~1\leq i,j\leq 2m\}$.
The Lie subalgebra $H(2m)^{(2)}$ of $H(2m)$ is simple and restricted, and we call it a \textit{Hamiltonian algebra}.

From now on we will (by abuse of notation) write $W(n)$,
$S(n)$ and $H(n)$ for the corresponding simple derived subalgebra,
with the convention that $n = 2m$ for the Hamiltonian type.

Suppose that $\fg=X(n), $ where $ X\in\{W,S,H\}$. By definition, it possesses a \textit{restricted $\mathbb{Z}$-grading}
\begin{align}\label{Z-grading}
\fg=\bigoplus\limits_{i=-1}^s\fg_i,\quad [\fg_i,\fg_j]\subseteq\fg_{i+j},\quad\fg_i^{[p]}\subseteq\fg_{pi}, \quad s\geq 1.
\end{align}
Given such an algebra $\fg$,
we consider the associated descending filtration $(\fg_{(i)})_{i\geq -1}$,
defined via
\begin{align}\label{natural filtration}
\fg_{(i)}:=\sum\limits_{j\geq i}\fg_j.
\end{align}

\subsection{Automorphism groups}\label{section automorphism groups}
Let us gather some facts on automorphisms.
It is well known that we have an isomorphism $\Aut(A(n))\cong\Aut_p(W(n))$;
$\varphi\mapsto\sigma_{\varphi}$,
given by
\begin{align}\label{assignment}
\sigma_{\varphi}(\p)=\varphi\circ\p\circ\varphi^{-1}
\end{align}
for all $\p\in W(n)$.
If $\fg\in\{W,S,H\}$,
then the group $\Aut_p(\fg)$ is connected, i.e. $G_\fg=\Aut_p(\fg)$,
and we have
$$G_\fg\cong\{g\in\Aut_p(W(n));~g(\fg)\subseteq \fg\}.$$
Moreover,
the group $G_\fg$ is a semidirect product $G_0\ltimes U$,
where $G_0$ consists of those
automorphisms preserving the $\mathbb{Z}$-grading (\ref{Z-grading}) of $\fg$ and $U$ is the unipotent radical (\cite{W}).
It is a consequence of the semidirect product decomposition that
\begin{align}\label{G-stable filtration}
g(\fg_{(i)})=\fg_{(i)}
\end{align}
for every $g\in G_\fg$ and $i\in\mathbb{Z}$.

Recall that the Poisson Lie algebra structure on $A(2m)$ is given by $\{f,g\}=D_H(f)(g)$ for all $f, g\in A(2m)$ (cf. \cite[Section 4.2]{St}),
where the linear map $D_H$ is defined by
$$D_H: A(2m)\rightarrow W(2m);~f\mapsto\sum_{i=1}^{2m}\sigma(i)\p_i(f)\p_{i'}.$$

For ease of reference we record the following well-known result:
\begin{Lemma}\cite[Theorem 7.3.4, 7.3.6]{St}\label{automorphism groups}
Let $\varphi\in\Aut(A(n))$.
Then $\sigma_\varphi$ induces an automorphism of $S(n)$ if and only if
\begin{align}\label{automorphism Sn}
\det(\p_i\varphi(x_j))\in k\smallsetminus\{0\},
\end{align}
and $\sigma_\varphi$ induces an automorphism of $H(n)$ if and only if
\begin{align}\label{auto of H}
\{\varphi(x_i), \varphi(x_j)\}=a\sigma(i)\delta_{i',j}
\end{align}
for some $a\in k\smallsetminus\{0\}$ and all $i, j$.
\end{Lemma}

\subsection{Maximal tori}\label{section maximal tori}
Given a restricted Lie algebra $(\fg,[p])$,
we let $\mu(\fg)$ denote the maximal dimension of all tori $\ft\subseteq\fg$ and let
$$\Tor(\fg):=\{\ft\subseteq\fg;~\ft~\text{torus},~\dim_k\ft=\mu(\fg)\}$$
be the set of tori of maximal dimension.

In the case of restricted Cartan type Lie algebras,
every maximal torus has maximal dimension (cf. \cite[Section 7.5]{St}).
Assume that $\fg\in\{W,S,H\}$.
Since the natural filtration (\ref{natural filtration}) is stable under the action of $G_\fg$,
it follows that the function
$$\chi_0: \Tor(\fg)\rightarrow\NN_0;~\ft\mapsto\dim_k(\ft\cap\fg_{(0)})$$
is constant on the $G_\fg$-orbits of $\Tor(\fg)$.
As shown by Demushkin in \cite{De1, De2},
there are exactly $\mu(\fg)+1$ orbits $\cO_0,\dots,\cO_{\mu(\fg)}$ in $\Tor(\fg)$ under the $G_\fg$-action,
and these have the following description:
$$\cO_r=\{\ft\in\Tor(\fg);~\chi_0(\ft)=\mu(\fg)-r\}.$$
For each of the three Cartan types we have canonical orbit representatives $\ft_{\fg,r}$ of $\cO_r$
given by
\begin{align}\label{tori direct sum}
\ft_{\fg,r}=\ft_{\fg,r}'\oplus\ft_{\fg,r}''
\end{align}
where
\begin{align}\label{tori in W}
\ft_{\fg,r}'=\sum\limits_{i=1}^r k(1+x_i)\p_i\quad{\rm and}\quad\ft_{\fg,r}''=\sum\limits_{i=r+1}^nkx_i\p_i,\quad\quad~{\rm if}~\fg=W(n)
\end{align}
\begin{align}\label{tori in S}
\ft_{\fg,r}'=\sum\limits_{i=1}^r k((1+x_i)\p_i-x_n\p_n)\quad{\rm and}\quad\ft_{\fg,r}''=\sum\limits_{i=r+1}^{n-1}k(x_i\p_i-x_n\p_n),\quad\quad~{\rm if}~\fg=S(n)
\end{align}
\begin{align}\label{tori in H}
\ft_{\fg,r}'=\sum\limits_{i=1}^r k((1+x_i)\p_i-x_{i'}\p_{i'})\quad{\rm and}\quad\ft_{\fg,r}''=\sum\limits_{i=r+1}^{m}k(x_i\p_i-x_{i'}\p_{i'}),\quad\quad~{\rm if}~\fg=H(2m)
\end{align}

\subsection{Weyl groups of $W,S,H$}\label{section weyl group}
Let $(\fg,[p])$ be a restricted Lie algebra with connected automorphism group $G_\fg$,
$\ft\subseteq\fg$ be a torus and we let $\N_{G_\fg}(\ft)$ and $\C_{G_\fg}(\ft)$ be the normalizer and the centralizer of $\ft$ in $G_\fg$,
respectively.
The group
$$W(\fg,\ft):=\N_{G_\fg}(\ft)/\C_{G_\fg}(\ft)$$
is referred to as the \textit{Weyl group of $\fg$ relative to $\ft$}.

For the three Cartan types $W,S$ and $H$,
we are interested in the Weyl group relative to maximal torus.
Let $\ft\in\Tor(\fg)$.
Since $W(\fg,g.\ft)\cong W(\fg,\ft)$ for every $g\in G_\fg$,
the Weyl group $W(\fg,\ft)$ only depends on the orbit $G_\fg.\ft\subseteq\Tor(\fg)$.
We may hence assume that $\ft=\ft_{\fg,r}$ (\ref{tori in W}),
(\ref{tori in S}),
(\ref{tori in H}) for some $0\leq r\leq\mu(\fg)$.

The following result was proved by Jensen in \cite{Je},
Prop. 3.6:
\begin{Proposition}\label{Jesen's result}
Assume $\fg\in\{W,S,H\}$.
Then
$W(\fg,\ft_{\fg,r})\cong(W_1\times W_2)\ltimes W_3$,
with
\begin{align}\label{Weyl group W_1}
W_1\cong\left\{
\begin{aligned}
&S_{n-r}&\quad {\rm if}~\fg=W(n) \\
&S_{n-r}& {\rm if}~\fg=S(n)\\
&S_{m-r}\ltimes\mathbb{Z}_2^{m-r}&\quad{\rm if}~\fg=H(n)
\end{aligned}
\right.
\end{align}
\begin{align}\label{Weyl group W_2}
W_2\cong\GL_{r}({\FF_p})
\end{align}
\begin{align}\label{Weyl group W_3}
W_3\cong\Mat_{r,\mu(\fg)-r}(\FF_p)
\end{align}
\end{Proposition}

\section{Semisimple orbits in $W,S,H$ }
\subsection{Semisimple elements in the standard tori}\label{semisimple elements in standard tori}
Assume that $\fg\in\{W,S,H\}$.
If $x\in S_\fg$ is a semisimple element,
then $\langle x\rangle_p$ denotes the torus generated by $x$.
We define a function
\begin{align}\label{index formula}
\Ind_{\fg}: S_{\fg}\rightarrow\NN_0;~x\mapsto\dim_k\langle x\rangle_p/(\langle x\rangle_p\cap\fg_{(0)}).
\end{align}
The {\it index} of an element $x$ is defined as $\Ind_{\fg}(x).$
In view of Section \ref{section maximal tori},
we have $\Ind_\fg(S_\fg)=\{r\in\NN_0;~0\leq r\leq\mu(\fg)\}$.
Given $r\in\{0,1,\dots,\mu(\fg)\}$,
it follows from (\ref{G-stable filtration}) that each fiber $\Ind_\fg^{-1}(r)$ is $G_\fg$-stable.
Clearly, $S_\fg$ is the disjoint union of all fibers, i.e.,
$$S_\fg=\bigcup\limits_{r=0}^{\mu(\fg)}\Ind_\fg^{-1}(r).$$
Indeed,
$\dim_k\ft_{\fg,r}/(\ft_{\fg,r}\cap\fg_{(0)})=r$ implies that
\begin{align}\label{tgrindexleq r}
\Ind_\fg(x)\leq r\quad {\rm for~all}~x\in\ft_{\fg,r}.
\end{align}
We denote by $\ft_{\fg,r}^{r}:=\ft_{\fg,r}\cap\Ind_\fg^{-1}(r)$ the set of those elements of $\ft_{\fg,r}$ whose index is $r$.
If $r=0$, then (\ref{tgrindexleq r}) yields $\ft_{\fg,0}^{0}=\ft_{\fg,0}$.
Note that the dimension $\Ind_\fg(x)$ does not change when $x$ is replaced by its $G_\fg$-conjugate.
Observing (\ref{tori in W}), (\ref{tori in S}) and (\ref{tori in H}),
we conclude that $\Ind_\fg^{-1}(0)$ is just the $G_\fg$-saturation $G_\fg.\ft_{\fg,0}$.

We put $y_i:=x_i+1$.
In order to describe the set $\Ind_\fg^{-1}(r)$,
we introduce the following notations:
$$
\ft_{\fg,r}=\langle d_1^{\fg,r},\ldots,d_{\mu(\fg)}^{\fg,r}\rangle, \quad 0\leq r\leq\mu(\fg).
$$
\begin{equation}\label{standard basis W}
d_i^{\fg,r}:=\left\{
\begin{aligned}
y_i\p_i,~& 1\leq i\leq r, \\
x_i\p_i,~& r+1\leq i\leq n.
\end{aligned}
\right.~\quad\quad~{\rm if}~\fg=W(n)
\end{equation}
\begin{equation}\label{standard basis S}
d_i^{\fg,r}:=\left\{
\begin{aligned}
y_i\p_i-x_n\p_n,~& 1\leq i\leq r, \\
x_i\p_i-x_n\p_n,~& r+1\leq i\leq n-1.
\end{aligned}
\right.~\quad\quad~{\rm if}~\fg=S(n)
\end{equation}
\begin{equation}\label{standard basis H}
d_i^{\fg,r}:=\left\{
\begin{aligned}
y_i\p_i-x_{i'}\p_{i'},~& 1\leq i\leq r, \\
x_i\p_i-x_{i'}\p_{i'},~& r+1\leq i\leq m.
\end{aligned}
\right.~\quad\quad~{\rm if}~\fg=H(2m)
\end{equation}

The following lemma is well-known. We provide a proof here for convenience.
\begin{Lemma}\label{3.1}
Given a restricted Lie algebra $(\fl,[p])$, we let $\ft\subseteq\fl$ be a torus with basis $\{t_1,\dots,t_n\}$ such that $t_i^{[p]}=t_i$ for $1\leq i\leq n$.
If $t=\sum_{i=1}^{n}\lambda_it_i\in\ft$, then $\ft=\langle t, t^{[p]},\cdots,t^{[p]^{n-1}}\rangle $ if and only if $\lambda_1,\ldots,\lambda_n$ are $\FF_p$-linearly independent.

\end{Lemma}
\begin{proof}

It is clear that $\{t, t^{[p]},\cdots,t^{[p]^{n-1}}\}$ is linearly independent if and only if $\det( (\lambda_i ^{p^{j-1}})_{1\leq i,j\leq n} )\neq 0$.
Since $ \det( (\lambda_i ^{p^{j-1}})_{1\leq i,j\leq n} )= \prod_{i=1}^n \prod_{a_1,\cdots, a_{{i-1}} \in\FF_p} (a_1\lambda_1+\cdots+a_{i-1}\lambda_{i-1}+\lambda_i) $ (see \cite[Section 2]{Di} for example), we have $ \det((\lambda_i ^{p^{j-1}})_{1\leq i,j\leq n} )\neq 0 $ if and only if $\lambda_1,\ldots,\lambda_n$ are $\FF_p$-linearly independent.
\end{proof}

\begin{Lemma}\label{key lemma}
Keep the notations as above.
Let $1\leq r\leq \mu(\fg)$ and $\ft_{\fg,r}$  be the standard maximal torus with basis $\{d_1^{\fg,r},\ldots,d_{\mu(\fg)}^{\fg,r}\}$.
Suppose that $d=\sum_{i=1}^{\mu(\fg)}\lambda_id_i^{\fg,r}\in\ft_{\fg,r}$.
Then $d\in\ft_{\fg, r}^r$ if and only if $\lambda_1,\ldots,\lambda_r$ are $\FF_p$-linearly independent.
\end{Lemma}
\begin{proof}
If $\lambda_1,\dots,\lambda_r$ are linearly independent over the prime field $\FF_p$,
then the coset of $d$ modulo $\ft_{\fg,r}\cap\fg_{(0)}$ generates an $r$ dimensional torus by Lemma \ref{3.1},
so that $\Ind_\fg(d)=\dim_k\langle d\rangle_p/(\langle d\rangle_p\cap\fg_{(0)})=r$, i.e. $d\in\ft_{\fg, r}^r$.

Conversely, if $\lambda_1,\dots,\lambda_r$ are $\FF_p$-linearly dependent,
then
$$d\equiv \alpha_1t_1+\cdots+\alpha_mt_m~(\mathrm{mod}~\ft_{\fg,r}\cap\fg_{(0)})$$
where $m<r$, $\alpha_i\in k$ and each $t_i$ is a linear combination of $d_1^{\fg,r},\dots,d_{\mu(\fg)}^{\fg,r}$
with coefficients in $\FF_p$.
In this case
$$\langle d\rangle_p\subseteq kt_1+\cdots kt_m+\fg_{(0)},$$
and it is immediately clear that $\Ind_\fg(d)<r$.
\end{proof}

\begin{Lemma}\label{r-1lemma}
Suppose that $d=\sum_{i=1}^{\mu(\fg)}\lambda_id_i^{\fg,r}\in\ft_{\fg,r}$.
If $\lambda_1,\dots,\lambda_r$ are $\FF_p$-linearly dependent,
then there exists $\sigma_\varphi\in G_\fg$ such that $\sigma_\varphi(d)\in\ft_{\fg,r-1}$.
\end{Lemma}
\begin{proof}
By assumption,
there exist $u_i\in\FF_p$ such that $\sum_{i=1}^ru_i\lambda_i=0$.
We may assume without loss of generality that $u_r=-1$,
so that $\lambda_r=\sum_{i=1}^{r-1}u_i\lambda_i$.
We put $(u):=(u_1,\dots,u_{r-1})$ and define $y^{(u)}:=y_1^{u_1}\cdots y_{r-1}^{u_{r-1}}$.

Assume first that $\fg=W(n)$.
We define an automorphism of $A(n)$ by
\[ \varphi(x_i)=\left\{
	\begin{tabular}{ll}
	$ x_i $ & if~$ i\neq r,$\\
	$ x_r+y^{(u)}-1 $  & if~$ i=r.$
\end{tabular}\right. \]
Using (\ref{assignment})(see also the formula in \cite[p. 234]{De1})) one can show by direct computation that
$$\sigma_\varphi(d)=\sum\limits_{i=1}^{r-1}\lambda_iy_i\p_i+\sum\limits_{i=r}^n\lambda_ix_i\p_i\in\ft_{W,r-1}.$$

Assume now $\fg=S(n)$.
Define $\varphi\in\Aut(A(n))$ by
\[ \varphi(x_i)=\left\{
	\begin{tabular}{ll}
	$ x_i $ & if~$ i\neq r,$\\
	$ x_r+y^{(u)}-1 $  & if~$ i=r.$
\end{tabular}\right. \]
As $\det(\p_i(\varphi(x_j)))=1$,
Lemma \ref{automorphism groups} ensures that $\sigma_\varphi\in G_\fg$.
Then we have:
$$\sigma_\varphi(d)=\sum\limits_{i=1}^{r-1}\lambda_i(y_i\p_i-x_n\p_n)+\sum\limits_{i=r}^{n-1}\lambda_i(x_i\p_i-x_n\p_n)\in\ft_{S,r-1}.$$

Finally the case $\fg=H(n)$,
with $n=2m$ for some $m\geq 1$.
Using multi-index notation (see \cite[Section 2.1]{St}),
we define
\[\varphi(x_i)=\left\{
		\begin{tabular}{ll}
		$ x_i $ & if $ i\neq r$ and $1\leq i\leq m$,\\
		$ x_r+y^{(u)}-1$  & if $i=r$,\\
		$ x_i-u_{i'}y^{(u)-\epsilon_{i'}}x_{m+r} $ & if $ m+1\leq i\leq m+r-1$,\\
		$ x_i $ & if $ m+r\leq i\leq 2m. $
		\end{tabular}\right. \]
Observing $\det(\p_i(\varphi(x_j)))=1$,
we thus obtain $\varphi\in\Aut(A(n))$.
Moreover, we claim that
$$\{\varphi(x_i), \varphi(x_j)\}=\sigma(i)\delta_{i',j} \quad {\rm for ~all}~i, j.$$
We just deal with $i=r$,
the other cases are similar.
Suppose that $1\leq j\leq m$.
Since both $\varphi(x_r)$ and $\varphi(x_j)$ lie in the algebra $A(m)$,
it follows that $\{\varphi(x_r), \varphi(x_j)\}=0$.
If $m+1\leq j\leq m+r-1$,
then
\[\{\varphi(x_r), \varphi(x_j)\} =\{ x_r+y^{(u)}-1, x_j-u_{j'}y^{(u)-\epsilon_{j'}}x_{m+r} \}\]
		\[=\sigma(r)\p_r(x_r)\p_{m+r}(-u_{j'}y^{(u)-\epsilon_{j'}} x_{m+r}) + \sigma({j'})\p_{j'}(y^{(u)})\p_j(x_j)=0.\]
Now for $m+r\leq j\leq 2m$,
so that $\{\varphi(x_r), \varphi(x_j)\} =\{ x_r+y^{(u)}-1, x_j\}=\delta_ {r',j}$.
This establishes our claim.
Consequently,
Lemma \ref{automorphism groups} implies that $\sigma_\varphi\in G_\fg$.
By the same token,
we have
$$\sigma_\varphi(d)=\sum\limits_{i=1}^{r-1}\lambda_i(y_i\p_i-x_{i'}\p_{i'})+\sum\limits_{i=r}^{m}\lambda_i(x_i\p_i-x_{i'}\p_{i'})\in\ft_{H,r-1}.$$
\end{proof}

\begin{Corollary}\label{corollary}
Assume $\fg\in\{W,S,H\}$ and let $x\in S_\fg$ be a semisimple element of $\fg$.
Then
\begin{align}\label{index=mini}
\Ind_\fg(x)=\min\{r;~G_\fg.x\cap\ft_{\fg,r}\neq\emptyset,~0\leq r\leq\mu(\fg)\}.
\end{align}
Let $r\in\{0,1,\dots,\mu(\fg)\}$.
In particular, up to conjugacy,
there exists a unique maximal torus $\ft$ such that $\ft\cap\Ind_\fg^{-1}(r)\neq\emptyset$ and $\Ind_\fg(x)\leq r$ for all $x\in\ft$.
\end{Corollary}
\begin{proof}
We put $l:=\min\{r;~G_\fg.x\cap\ft_{\fg,r}\neq\emptyset,~0\leq r\leq\mu(\fg)\}$.
we may assume that $x\in \ft_{\fg,l}$.
Now Lemma \ref{key lemma} and Lemma \ref{r-1lemma} in conjunction with the minimality of $l$ yields $x\in\ft_{\fg,l}^{l}$.
Consequently,
$\Ind_\fg(x)=l$.

To verify the last assertion,
we note that $\Ind_\fg(x)\leq r$ for all $x\in\ft_{\fg,i}$ and $i\leq r$ (\ref{tgrindexleq r}).
In view of (\ref{index=mini}),
we obtain $\ft_{\fg,i}\cap\Ind_\fg^{-1}(r)=\emptyset$ whenever $i<r$.
This proves the uniqueness.
\end{proof}

\subsection{Semisimple orbits}
In this section,
we turn to the set $\Ind_\fg^{-1}(r)/G_\fg$
for the restricted Cartan type Lie algebra $\fg\in\{W,S,H\}$.
We have seen in the foregoing section that the set $\Ind_\fg^{-1}(r)$ coincides with the $G_\fg$-saturation $G_\fg.\ft_{\fg,r}^r$.
It will be necessary to consider the $G_\fg$-conjugacy relation on the set $\ft_{\fg,r}^r$.

\begin{Lemma}\label{keylemma2}
Assume that $\fg\in\{W,S,H\}$ with connected automorphism group $G_\fg$.
Let $d, t\in\ft_{\fg,r}^r$.
If $d$ and $t$ are conjugate under $G_\fg$,
then there exists $\sigma_\psi\in\N_{G_\fg}(\ft_{\fg,r})$ such that $\sigma_\psi(d)=t$.
\end{Lemma}
\begin{proof}
Let $\sigma_\varphi\in G_\fg$ be such that $\sigma_\varphi(d)=t$.
We write $d=\sum_{i=1}^{\mu(\fg)}\beta_id_i^{\fg, r}$ as well as $t=\sum_{i=1}^{\mu(\fg)}\alpha_id_i^{\fg, r}$.
Define
\[z_i=\left\{
	\begin{tabular}{ll}
	$ y_i $ & if $ i=1,\cdots, r, $\\
	$ x_i $ & otherwise.
	\end{tabular}
	\right.
	\text{and}\quad
f_i=\varphi(z_i).\]
Setting $\beta=(\beta_1,\cdots,\beta_{\mu(\fg)})$ and $\alpha=(\alpha_1,\cdots,\alpha_{\mu(\fg)})$,
we apply (\ref{assignment}) in conjunction with (\ref{standard basis W})-(\ref{standard basis H}) to see that
\begin{align}\label{weightvector W}
t\cdot(f_i)=\sigma_\varphi(d)\cdot(f_i)=(\varphi\circ d\circ\varphi^{-1})(\varphi(z_i))
=\varphi(\beta_iz_i)=\beta_if_i,~\quad~1\leq i\leq n,\quad\fg=W(n).
\end{align}
\begin{align}\label{weightvector S}
t\cdot(f_i)=\left\{
\begin{aligned}
\beta_if_i,~1\leq i\leq n-1,\\
(-\sum\limits_{j=1}^{n-1}\beta_j)f_i,~i=n.
\end{aligned}
\right.~\quad\fg=S(n).
\end{align}
\begin{align}\label{weightvector H}
t\cdot(f_i)=\left\{
\begin{aligned}
\beta_if_i,~&1\leq i\leq m,\\
-\beta_{i'}f_i,~&m+1\leq i\leq 2m.
\end{aligned}
\right.~\quad\fg=H(2m).
\end{align}
As a result,
$f_i$ is a weight vector with respect to the canonical action of $t$ on $A(n)$.
We claim that

\medskip
(\textdagger)\quad there exists matrices $A, B, \tau$ such that $\beta=\alpha\left(\begin{matrix}A & B\\ 0 & \tau\end{matrix}\right)$,
where $ A=(a_{ij})\in \GL_r(\FF_p), B=(b_{ij})$ and $\tau\in W_1$ (\ref{Weyl group W_1}).

Suppose that $\fg=W(n)$.
Since $\varphi\in\Aut(A(n))$,
the weight space decomposition ensures that
\begin{align}
f_j~{\rm has~the~term}~z_1^{a_{1j}}\cdots z_r^{a_{rj}}~{\rm with~weight}~\sum_{i=1}^r\alpha_ia_{ij}\quad 1\leq j\leq r
\end{align}
\begin{align}
f_j~{\rm has~the~term}~c_jx_{\tau(j)}z_1^{b_{1j}}\cdots z_r^{b_{rj}}~{\rm with~weight}~
\sum_{i=1}^{r}\alpha_ib_{ij}+\alpha_{\tau(j)}\quad r+1\leq j\leq n
\end{align}
where $\tau$ is a permutation on $\{r+1,\dots,n\}$.

Assume now $\fg=S(n)$.
We put $\alpha_n:=-\sum_{j=1}^{n-1}\alpha_{j}$.
The same argument shows that
\begin{align}
f_j~{\rm has~the~term}~z_1^{a_{1j}}\cdots z_r^{a_{rj}}~{\rm with~weight}~\sum_{i=1}^r\alpha_ia_{ij}\quad 1\leq j\leq r
\end{align}
\begin{align}
f_j~{\rm has~the~term}~c_jx_{\tau(j)}z_1^{b_{1j}}\cdots z_r^{b_{rj}}~{\rm with~weight}~
\sum_{i=1}^{r}\alpha_ib_{ij}+\alpha_{\tau(j)}\quad r+1\leq j\leq n
\end{align}
where $\tau$ is a permutation on $\{r+1,\dots,n\}$.

Finally consider the case $\fg=H(2m)$.
As $\sigma_\varphi\in G_\fg$,
it follows that $f_j$ has the form (modulo the corresponding weight space) $z_1^{a_{1j}}\cdots z_r^{a_{rj}}$
with weight $\sum_{i=1}^r\alpha_ia_{i,j}$ for every $j\in\{1,\dots,r\}$.

For $r+1\leq j\leq m$,
combining (\ref{auto of H}) with (\ref{weightvector H}) one obtains that
$f_j$ has the term $z_1^{b_{1j}}\cdots z_r^{b_{rj}}x_{\tau(j)}$,
where $\tau$ is a permutation on $\{r+1,\dots,m,m+r+1,\cdots,2m\}$.
In view of (\ref{auto of H}),
$\{f_j, f_{j'}\}$ is a non-zero constant.
Direct computation shows that $\tau(j')=\tau(j)'$.
so that $\tau$ can be identified with an element of $S_{m-r}\ltimes\mathbb{Z}_2^{m-r}$,
where the copies of $\mathbb{Z}_2$ act by $'$.
Consequently,
$f_j$ has weight $\sum_{i=1}^r\alpha_ib_{ij}+\sigma(\tau(j))\alpha_{\omega(j)}$ with $(\omega,a)=\tau\in S_{m-r}\ltimes\mathbb{Z}_2^{m-r}$.

Thanks to Lemma \ref{key lemma},
both $\{\alpha_1,\cdots,\alpha_r\}$ and $\{\beta_1,\cdots,\beta_r\}$ are $\FF_p$-linearly independent.
It follows that $A\in\GL_r(\FF_p)$.
This proves (\textdagger).
Now, Proposition \ref{Jesen's result} ensures the existence of $\sigma_\psi$.
\end{proof}

\begin{Theorem}\label{main theorem}
Assume that $\fg\in\{W,S,H\}$.
Let $r\in\{0,1,\dots,\mu(\fg)\}$.
Then there is a bijective correspondence $$\Ind_\fg^{-1}(r)/G_\fg\rightarrow\ft_{\fg,r}^r/W(\fg,\ft_{\fg,r}).$$
\end{Theorem}
\begin{proof}
Let $x\in\Ind_\fg^{-1}(r)$.
Corollary \ref{corollary} readily yields $G_\fg.x\cap\ft_{\fg,r}\neq\emptyset$.
It follows that $\Ind_\fg^{-1}(r)$ coincides with the $G_\fg$-saturation $G_\fg.\ft_{\fg,r}^r$.
The assertion now follows from Lemma \ref{keylemma2}.
\end{proof}

\subsection{Quotients}
In this section, we would like to give a description of the quotients $\ft_{\fg,r}^r/W(\fg,\ft_{\fg,r})$ by employing $p$-polynomials.
Recall that a \textit{$p$-polynomial} is a polynomial of the form
$$f(t)=a_lt^{p^l}+a_{l-1}t^{p^{l-1}}+\cdots+a_0t\in k[t].$$
For each $ x\in \fg,$ define $ f(x):=a_lx^{[p]^l}+a_{l-1}x^{[p]^{l-1}}+\cdots+a_0x\in\fg$.
Let $\ft\subseteq\fg$ be a torus.
Given a semisimple element $x\in\ft$,
there exists a monic $p$-polynomial $f_x(t)$ of lowest degree such that $f_x(x) =0$.
It is unique and we call it the \textit{minimal $p$-polynomial} of $x$.
By general theory, the orbit of an element $x\in\ft$ with respect to the whole group $\Aut_p(\ft)$ is completely determined by the
minimal $p$-polynomial of $x$.

In the case $r=\mu(\fg)$,
according to Proposition \ref{Jesen's result} (see also \cite[Theorem 1]{Pre} and \cite[Theorem 5.3]{BFS}),
there is an isomorphism $W(\fg,\ft_{\fg,\mu(\fg)})\cong\GL_{\mu(\fg)}(\FF_p)$.
Let $x,y\in\ft_{\fg,\mu(\fg)}$.
The foregoing observation implies that $x$ and $y$ are in the same $W(\fg,\ft_{\fg,\mu(\fg)})$-orbit if and only if $f_{x}=f_{y}$.

For general $r$, let $x\in\ft_{\fg,r}$.
We denote by $\bar{x}$ the image of $x$ under the canonical projection $\pi:\ft_{\fg,r}\rightarrow\ft_{\fg,r}/(\ft_{\fg,r}\cap\fg_{(0)})$.
Let $f_{\bar{x}}(t)$ be the minimal $p$-polynomial of $\bar{x}$.
It follows that $f_{\bar{x}}(t)$ is the monic $p$-polynomial of smallest degree such that $f_{\bar{x}}(x)\in\fg_{(0)}$.

\begin{Proposition}\label{main prop}
Keep the notations as before.
Let $x,y\in\ft_{\fg,r}$.
Then $x$ and $y$ are in the same $W(\fg,\ft_{\fg,r})$-orbit if and only if $f_{\bar{x}}=f_{\bar{y}}$ and
$f_{\bar{x}}(x),f_{\bar{y}}(y)$ lie in the same $W_1$-orbit (\ref{Weyl group W_1}).
\end{Proposition}
\begin{proof}
Let $\sigma\in W(\fg,\ft_{\fg,r})$ be such that $\sigma(x)=y$.
The Weyl group $W(\fg,\ft_{\fg,r})$ leaves invariant the subtorus $\ft_{\fg,r}\cap\fg_{(0)}$.
It follows that $\sigma\in\Aut_p(\ft_{\fg,r}/(\ft_{\fg,r}\cap\fg_{(0)}))$ and $\sigma(\bar{x})=\bar{y}$.
Consequently, $f_{\bar{x}}=f_{\bar{y}}$.
Now, let $f_{\bar{x}}(t)=f_{\bar{y}}(t)=t^{p^l}+a_{l-1}t^{p^{l-1}}+\cdots+a_0t$.
We have
\begin{eqnarray*}
\sigma(f_{\bar{x}}(x))&=&\sigma(x^{p^l}+a_{l-1}x^{p^{l-1}}+\cdots+a_0x)\\
&=&\sigma(x)^{p^l}+a_{l-1}\sigma(x)^{p^{l-1}}+\cdots+a_0\sigma(x)=f_{\bar{y}}(y).
 \end{eqnarray*}
Since $W(\fg,\ft_{\fg,r})$ acts on the subtorus $\ft_{\fg,r}\cap\fg_{(0)}$ via the classical Weyl group $W_1$,
there exists an element $w\in W_1$ such that $\sigma(f_{\bar{x}}(x))=w(f_{\bar{x}}(x))=f_{\bar{y}}(y)$.

Conversely, denote $f:=f_{\bar{x}}=f_{\bar{y}}$. Let $\tau\in W_1$ be such that $\tau(f(x))=f(y)$.

We write $x=(x',x'')$ and $y=(y',y'')$ with $x',y'\in\ft_{\fg,r}'$ and $x'',y''\in\ft_{\fg,r}''$ (see (\ref{tori direct sum})).
It is easy to see that $ f $ is also the minimal $p$-polynomial of $x'$ and $y'$.
General theory provides an invertible matrix $A\in\GL_{r}(\FF_p)$ such that $x'A=y'$.
Since the $p$-polynomial is additive,
it follows that $\tau(f(x))=\tau(f(x'+x''))=\tau(f(x''))=f(y)=f(y'+y'')=f(y'')$.
As a result, $f(x'+y''-\tau(x''))=0$.
Hence, $f$ is the minimal $p$-polynomial of $x'+y''-\tau(x'')$.
By the same token, there exists an invertible matrix $\left(\begin{matrix}B & C\\ D & E\end{matrix}\right)\in\GL_{\mu(\fg)}(\FF_p)$ such that $(x',0)\left(\begin{matrix}B & C\\ D & E\end{matrix}\right)=(x',y''-\tau(x''))$.
Consequently, $(x',x'')\left(\begin{matrix}A& C\\ 0 & \tau\end{matrix}\right)=(y',y'')$,
and our assertion now follows
directly from Proposition \ref{Jesen's result}.
\end{proof}

\section{Normalizers and centralizers in $W(n)$}
In this section,
$\fg$ always denotes the $n$-th Jacobson-Witt algebra.
For convenience, we set $\ft_r:=\ft_{\fg,r},~0\leq r\leq n$ (see (\ref{tori in W})).
We shall compute $\N_{G_\fg}(\ft)$ and $\C_{G_\fg}(\ft)$ for every $\ft\in\Tor(\fg)$.
Up to conjugacy,
we may thus assume that $\ft=\ft_r$ for some $r\in\{0,1,\dots,n\}$.
It should be noted that the isomorphisms
\begin{align}\label{premet's result}
\N_{G_\fg}(\ft_n)\cong\GL_n(\FF_p)\quad{\rm and}\quad \C_{G_\fg}(\ft_n)\cong\{\id_{\ft_n}\}
\end{align}
were established by Premet (see \cite[p. 139]{Pre}).

Recall that $\ft_r=\sum_{i=1}^n kz_i\p_i$,
where $z_i=y_i$ for $1\leq i\leq r$ and $z_i=x_i$ for $r+1\leq i\leq n$ (\ref{standard basis W}).

\begin{Proposition}\label{normalizers of tr}
Assume $r\in\{0,\dots,n\}$.
Then
$$\N_{G_\fg}(\ft_r)=\{\varphi\in\Aut(A(n));~\varphi~{\rm has~form~}~(\ast)\}$$

($\ast$)
If $1\leq j\leq r$,
$\varphi(z_j)=\prod_{i=1}^rz_i^{m_{ij}}$,
where $(m_{ij})_{1\leq i,j\leq r}\in\GL_r(\FF_p)$.
If $ r+1\leq j\leq n$,
$\varphi(z_j)=a_jz_{\tau(j)}\prod_{l=1}^rz_l^{m_{lj}}$,
 where $a_j\in k^*,~\tau\in S_{n-r}, m_{lj}\in\FF_p$.

\end{Proposition}
\begin{proof}
Let $\varphi\in\Aut(A(n))$ be such that $\sigma_\varphi\in\N_{G_\fg}(\ft_r)$.
There exists an invertible matrix $(m_{ij})\in\GL_n(k)$ such that
$$z_i\p_i=\sum\limits_{j=1}^n m_{ij}\sigma_\varphi(z_j\p_j)\quad 1\leq i\leq n.$$
We put $f_j:=\varphi(z_j)$,
then it is a simple exercise in linear algebra to show
\begin{align}\label{formula}
z_i\p_if_j=m_{ij}f_j,\quad 1\leq i, j\leq n.
\end{align}	
Note that $\{z_1^{a_1}\cdots z_n^{a_n};~0\leq a_1,\dots,a_n<p\}$ is a $k$-basis of $A(n)$
consisting of weight vectors.
Our assertion follows from (\ref{formula}) in conjunction with the weight space decomposition.
\end{proof}
\begin{Corollary}
$\C_{G_\fg}(\ft_r)\cong (k^*)^{n-r}$.
In particular, $\C_G(\ft_r)=\N_G(\ft_r)^{\circ}$.
\end{Corollary}
\begin{proof}
As $\sigma_\varphi\in \C_{G_\fg}(\ft_r)$,
direct computation shows that the weights $z_i$ and $\varphi(z_i)$ are the same.
The assertion follows by applying the similar argument as in Proposition \ref{normalizers of tr}.
\end{proof}


\end{document}